\documentclass[reqno,12pt]{amsart}
\usepackage{amsmath,amsthm,enumerate,cite,graphics,pstricks,amsfonts,latexsym,amsopn,verbatim,amscd,amssymb}
\usepackage{color}
\usepackage{psfrag}
\usepackage[all]{xy}

\theoremstyle{plain}
\newtheorem{thm}{Theorem}[section]

\newtheorem{lem}[thm]{Lemma}
\newtheorem{cor}[thm]{Corollary}
\newtheorem{prop}[thm]{Proposition}

\newtheorem{problem}{Theorem}[section] 
\newtheorem{prob}[problem]{Problem}

\theoremstyle{definition}

\newtheorem{rem}[thm]{Remark}

\DeclareMathOperator{\Cent}{Cent}
\DeclareMathOperator{\I}{I}

\begin{document} 

\title[Characterizations of some groups in  terms of centralizers]{Characterizations of some groups in  terms of centralizers} 

\author[S. J. Baishya  ]{Sekhar Jyoti Baishya} 
\address{S. J. Baishya, Department of Mathematics, Pandit Deendayal Upadhyaya Adarsha Mahavidyalaya, Behali, Biswanath-784184, Assam, India.}

\email{sekharnehu@yahoo.com}

\begin{abstract}
A group $G$ is said to be $n$-centralizer if  its number of element centralizers $\mid \Cent(G)\mid=n$, an F-group if every non-central element centralizer contains no other element centralizer and a CA-group if  all non-central element centralizers are abelian. For any non-abelian $n$-centralizer group $G$, we  prove that $\mid \frac{G}{Z(G)}\mid \leq (n-2)^2$, if $n \leq  12$ and $\mid \frac{G}{Z(G)}\mid \leq 2(n-4)^{{log}_2^{(n-4)}}$  otherwise, which improves an earlier result.  We prove that if $G$ is an arbitrary non-abelian $n$-centralizer F-group, then gcd$(n-2, \mid \frac{G}{Z(G)}\mid) \neq 1$. For a finite F-group  $G$, we show that $\mid \Cent(G)\mid \geq \frac{\mid G \mid}{2}$ iff $G \cong A_4 $, an extraspecial $2$-group or a Frobenius group with abelian kernel and complement of order $2$.  Among other results, for a finite group $G$ with non-trivial center, it is proved that $\mid \Cent(G)\mid = \frac{\mid G \mid }{2}$ iff $G$ is an extraspecial $2$-group. We give a family of F-groups which are not CA-groups and  extend an earlier result.
\end{abstract}

\maketitle

\section{Introduction} \label{S:intro}

A group $G$ is said to be an F-group if for every $x, y \in G \setminus Z(G)$, $C(x) \leq C(y)$ implies that $C(x)=C(y)$, where $C(x)$ and $C(y)$ are centralizers of $x$ and $y$ respectively.
In 1953, Ito \cite {ito} introduced the notion of the class of finite F-groups and since then the influence of the element centralizers on the structure of groups has been studied extensively. An interesting subclass of finite F-groups is the class of I-groups, consisting of groups in which all centralizers of non-central elements are of same order. Ito in \cite {ito} proved that I-groups are nilpotent and direct product of an abelian group and a group of prime power order. Later on  in 2002, Ishikawa \cite{ish} proved that I-groups are of class at most $3$. In 1971,  Rebmann \cite {reb} investigated and classified F-groups.

In  1970,  Schmidt \cite {schmidt} introduced the notion of CA-groups (another important subclass of F-groups) consisting of groups in which all centralizers of non-central elements are abelian. He classified the finite CA-groups, whose structure is very much similar to that of finite F-groups. In 1973, the author  \cite{kos} studied finite groups in which the centralizer of every non-central element is a maximal subgroup. In 2009, the authors \cite{ctc095} introduced and characterised CH-groups, consisting of finite groups $G$ in which for every $x, y \in G \setminus Z(G)$, $xy=yx$ implies that $\mid C(x)\mid = \mid C(y)\mid$. Recently, in 2017 the author \cite{brough} studied finite groups in terms of central intersections of element centralizers. More recently, the present author in \cite{baishya2} used the influence of centralizers to characterise capable groups of order $p^2q$, $p \neq q$ are primes (recall that a group is said to be capable if it is the central factor of some group).

Starting with Belcastro and Sherman  \cite{ctc092} in 1994, the characterization of groups in terms of   the number of element centralizers (denoted by $\mid \Cent(G)\mid$) have been considered by many researchers (see for example \cite{zarrin094, en09, baishya2, baishyaF, con,  jaa4, zarrin0942} for finite groups and \cite{non,  zarrin09422} for infinite groups). In this paper,  for any $n$-centralizer non-abelian group $G$, we prove that if $n>12$, then   $\mid \frac{G}{Z(G)}\mid \leq 2(n-4)^{{log}_2^{(n-4)}}$ and if $n \leq 12$, then $\mid \frac{G}{Z(G)}\mid \leq (n-2)^2$. This  improves \cite[Theorem B]{jaa4}. Among other results, we prove that if $G$ is an  $n$-centralizer non-abelian F-group, then gcd$(n-2, \mid \frac{G}{Z(G)}\mid) \neq 1$. For a finite group $G$ with trivial center and largest prime divisor $q$, we prove that $\mid \Cent(G)\mid \geq q+2$, with equality  iff $G=C_q \rtimes C_n$ is a Frobenius group. For a finite F-group  $G$, we prove that $\mid \Cent(G)\mid \geq \frac{\mid G \mid}{2}$ iff $G \cong A_4 $, an extraspecial $2$-group or a Frobenius group with abelian kernel and complement of order $2$.  It is also proved that if $G$ is a finite group with non-trivial center, then $\mid \Cent(G)\mid \leq \frac{\mid G \mid}{2}$, with equality iff $G$ is an extraspecial $2$-group. Finally, we conclude the paper with a family of F-groups which are not CA-groups and  extend an earlier result.

Throughout this paper, $G$ is a  group with center $Z(G)$, commutator subgroup $G'$ and the set of element centralizers $\Cent(G)$. We write $Z(x)$ to denote the center of the centralizer $C(x)$ and $C_n$ to denote the cyclic group of order $n$.

\section{The main  results}

We begin with the following elementary lemma.

\begin{lem}\label{np1}
Let $G$ be a  group and $x, y \in G$. Then $C(x) \subseteq C(y)$ iff $Z(y) \subseteq Z(x)$.
\end{lem}

\begin{proof}
Suppose $C(x) \subseteq C(y)$. If $a \in Z(y)$, then $a \in C(x)$ and consequently, $a \in Z(x)$. Conversely, if $Z(y) \subseteq Z(x)$, then  $y \in Z(x)$ and so $C(x) \subseteq C(y)$.
\end{proof}

\begin{cor}\label{co1}
Let $G$ be a  group and $x, y \in G$. Then $y \in Z(x)$ iff $Z(y) \subseteq Z(x)$.
\end{cor}

\begin{cor}\label{npcor1}
Let $G$ be a  non-abelian group such that    $\mid C(x) \mid=p$ ($p$ a prime) for some $x \in G$. Then $C(x) \subsetneq C(y)$ if and only if  $y=1$.
\end{cor}

\begin{proof}
If $C(x) \subsetneq C(y)$, then by Lemma \ref{np1}, $Z(y) \subsetneq Z(x)$, forcing  $y =1$. Converse is trivial. 
\end{proof}

\begin{lem}\label{np155}
Let $G$ be a  finite group and $x \in G $. Then $\mid \frac{C(x)}{Z(G)}\mid \leq \mid C(xZ(G))\mid \leq \mid C(x) \mid$, where $C(xZ(G))$ is the centralizer of $xZ(G)$ in $\frac{G}{Z(G)}$.
\end{lem}

\begin{proof}
The result follows from \cite[Lemma 1]{fong}, noting that $\frac{C(x)}{Z(G)} \leq C(xZ(G))$. 
\end{proof}

For a group $G$ and any $g, x \in G$, it is well known that $g^{-1}C(x)g=C(g^{-1}xg)$. One can easily verify the following analogues result.

\begin{lem}\label{z-class1}
Let $G$ be a  group and $x \in G$. Then $g^{-1}Z(x)g = Z(g^{-1}xg)$ for any $g \in G$.
\end{lem}

The following key result gives a characterization of an $n$-centralizer F-group. 
Recall that a group $G$ is said to be $n$-centralizer if  its number of element centralizers $\mid \Cent(G)\mid=n$.  A cover for  a group $G$ is a collection of non-trivial subgroups whose union is $G$ itself. A collection $\Pi $ of non-trivial subgroups of a group $G$ is called a  partition if every non-trivial element of $G$ belongs to a unique subgroup in $\Pi$.  A partition $\Pi$ of a group $G$ is said to be non-trivial if $\mid \Pi \mid \neq 1$ and normal if $g^{-1}Xg \in \Pi$ for every $X \in \Pi$ and $g \in G$.

\begin{prop}\label{z-class5}
A non-abelian group $G$ is an $n$-centralizer F-group if and only if $\Pi= \lbrace \frac{Z(x_i)}{Z(G)} \mid 1 \leq i \leq n-1 \rbrace$ is a non-trivial normal partition of $\frac{G}{Z(G)}$, where $ \lbrace Z(x_i) \mid 1 \leq i \leq n-1 \rbrace$ is the set of all centers of the proper centralizers of $G$.
\end{prop}

\begin{proof}
Let $G$ be a non-abelian $n$-centralizer F-group. Suppose $x, y \in G \setminus Z(G)$ such that $C(x) \neq C(y)$. Let $z \in (Z(x) \cap Z(y)) \setminus Z(G)$. Then $C(x), C(y) \subseteq C(z)$ and consequently, $C(x)=C(z)=C(y)$, which is a contradiction. Hence $\Pi$ is a  non-trivial partition of $\frac{G}{Z(G)}$. In the present scenario, in view of Lemma \ref{z-class1}, 
we have $g^{-1}Z(G)\frac{Z(x_i)}{Z(G)}gZ(G)=\frac{g^{-1}Z(x_i)g}{Z(G)}=\frac{Z(g^{-1}x_ig)}{Z(G)}$ for every $\frac{Z(x_i)}{Z(G)} \in \Pi$ and any $gZ(G) \in \frac{G}{Z(G)}$. Hence $\Pi$ is a non-trivial normal partition.

Conversely, suppose $\Pi$ is a partition of $\frac{G}{Z(G)}$. Then $Z(x_i) \nsubseteq Z(x_j)$ for any $i, j \in  \lbrace 1, \dots, n-1 \rbrace$, $i \neq j$. Therefore using Lemma \ref{np1}, $G$ is an F-group.
\end{proof}

As an immediate application of the above result, we also have the following characterization of an arbitrary F-group.

\begin{prop}\label{1np}
If $G$ is a non-abelian $n$-centralizer F-group, then gcd $(n-2, \mid \frac{G}{Z(G)}\mid) \neq 1$.
\end{prop}

\begin{proof}
Let $C(x_1), C(x_2), \dots, C(x_{n-1})$ be the proper centralizers of $G$. In view of  \cite[Theorem 1.2 (Baer)]{tom}, $\frac{G}{Z(G)}$ is finite by noting that $G=Z(x_1)\cup \dots \cup Z(x_{n-1})$. Moreover, since $G$ is an F-group,  using Proposition \ref{z-class5}, $ \lbrace \frac{Z(x_i)}{Z(G)} \mid 1 \leq i \leq n-1 \rbrace$ is a non-trivial normal partition of $\frac{G}{Z(G)}$. Therefore by \cite[Theorem 4.1]{mfdg}, we have the result.
\end{proof}

\begin{cor}\label{np22}
Let $G$ be a non-abelian  $n$-centralizer nilpotent F-group. Then $p \mid (n-2)$, where $\mid \frac{G}{Z(G)}\mid=p^k$, $p$ a prime.
\end{cor}

\begin{proof}
In view of  \cite[Theorem 1.2 (Baer)]{tom}, $\frac{G}{Z(G)}$ is finite by noting that $G=Z(x_1)\cup \dots \cup Z(x_{n-1})$. In the present scenario, by Proposition  \ref{z-class5} and \cite[Theorem 1.3]{mfdg}, we have $\frac{G}{Z(G)}$ is a $p$-group. Now, the result follows from Proposition \ref{1np}.
\end{proof}

Following Ito \cite{ito}, a finite group $G$ is said to be of conjugate type $(m, 1)$ if every proper centralizer of $G$ is of index $m$. He proved that a group of  conjugate type $(m, 1)$ is nilpotent and $m=p^k$ for some prime $p$. Moreover, he also proved that a group of  conjugate type $(p^k, 1)$ is a direct product of a $p$-group of the same type and an abelian group. The author in \cite{ish1} classified finite $p$-groups of conjugate type $(p, 1)$ and $(p^2, 1)$ up to isoclinism (for basic notions of isoclinism see \cite{pL95, non}). As an application to Proposition \ref{1np}, we also have the following result:

\begin{cor}\label{np2}
Let $G$ be a finite $n$-centralizer group of conjugate type $(m, 1)$.  Then $p \mid (n-2)$, where $m=p^k$, $p$ a prime.
\end{cor}

\begin{proof}
In view of Ito \cite{ito}, $ \frac{G}{Z(G)}$ is a $p$-group. Now, the result follows from  Proposition \ref{1np} by noting that $G$ is an F-group.
\end{proof}

 The author in \cite[Theorem B]{jaa4} proved that every arbitrary $n$-centralizer group is center-by-(finite of order $\leq$ max $\lbrace (n-2)^2, 2(n-3)^{{log}_2^{(n-3)}}\rbrace$). We improve this result as follows: [note that $(n-4)^{{log}_2^{(n-4)}} < (n-3)^{{log}_2^{(n-3)}}$].

\begin{thm}\label{bc1}
Let $G$ be a non-abelian $n$-centralizer group. 

\begin{enumerate}
	\item If $G$ is an F-group, then  $ \mid\frac{G}{Z(G)}\mid  \leq (n -2)^2$. In particular, if $ \mid\frac{Z(x)}{Z(G)} \mid  < \surd \mid\frac{G}{Z(G)} \mid$ for all $x \in G \setminus Z(G)$, then $ \mid\frac{G}{Z(G)}\mid  < (n -2)^2$.
	\item If $G$ is not an F-group, then $\mid\frac{G}{Z(G)}\mid \leq max \lbrace(n-3)^2, 2(n-4)^{{log}_2^{(n-4)}}\rbrace$.
\end{enumerate} 
\end{thm}

\begin{proof}
Let $C(x_1), C(x_2), \dots, C(x_{n-1})$ be the proper centralizers of $G$. In view of  \cite[Theorem 1.2 (Baer)]{tom}, $\frac{G}{Z(G)}$ is finite by noting that $G=Z(x_1)\cup \dots \cup Z(x_{n-1})$.\\ 

a) Now, if $G$ is an F-group,  using Proposition \ref{z-class5}, $\lbrace \frac{Z(x_i)}{Z(G)} \mid 1 \leq i \leq n-1 \rbrace$ is a partition of $\frac{G}{Z(G)}$. Therefore by
 \cite[Corollary 1]{par}, $ \mid\frac{G}{Z(G)}\mid  \leq (n -2)^2$. Second part follows from the proof of \cite[Corollary 1]{par}.\\

b) If $G$ is not an F-group, then in view of Lemma \ref{np1}, $Z(x_i) \subsetneq Z(x_j)$ for some $i, j \in  \lbrace 1, \dots, n-1 \rbrace$, $i \neq j$. Consequently, $G$ is covered by  $ \leq (n-2)$ abelian subgroups. In the present scenario, one can verify that $G$ has an irredundant covering by $ \leq (n-2)$ maximal abelian subgroups. Therefore by \cite[Theorem 4.2]{tom} we have the result.
\end{proof}

It is easy to see that every partition of a group is a cover but the converse is not true in general. According to the proof of  \cite[Theorem B]{jaa4}, for any $n$-centralizer non-abelian group $G$ the set $\lbrace \frac{Z(x_i)}{Z(G)} \mid 1 \leq i \leq n-1 \rbrace$ is a cover for $\frac{G}{Z(G)}$, but for such an F-group the same set becomes a partition for $\frac{G}{Z(G)}$. For more informations on related concepts see also \cite{smja}.

\begin{cor}\label{1sb}
Every arbitrary $n$-centralizer non-abelian group is 
 \[ 
 center-by-(finite \;\; of \;\; order \leq max \lbrace (n-2)^2, 2(n-4)^{{log}_2^{(n-4)}}\rbrace).
 \]
\end{cor}

\begin{cor}\label{5sb}
Let $G$ be a finite $n$-centralizer group of conjugate type $(p, 1)$, $p$ a prime. Then  $n-2=p$ iff $\frac{G}{Z(G)} \cong C_p \times C_p$ (compare with \cite[Lemma 3.1]{ctc09}).
\end{cor}

\begin{proof}
In view of Theorem \ref{bc1}, we have $\frac{G}{Z(G)} \cong C_p \times C_p$. Conversely, if $\frac{G}{Z(G)} \cong C_p \times C_p$, then $n=p+2$ by noting that in the present scenario $C(a) \cap C(b)=Z(G)$ for any $a, b \in G \setminus Z(G)$ with $ab \neq ba$ and each proper centralizer of $G$ contains exactly $p$ distinct right cosets of $Z(G)$.
\end{proof}

The following result improves \cite[Theorem 3.3]{en09}.

\begin{prop}\label{52sb}
Let $G$ be a finite $n$-centralizer group of conjugate type $(p^2, 1)$,  $p$ a prime. Then  $n-2=p^2$ iff $\frac{G}{Z(G)} \cong C_p \times C_p \times C_p \times C_p$. 
\end{prop}

\begin{proof}
In view of Theorem \ref{bc1},  \cite[Proposition 2.14]{baishya2} and  \cite[Proposition 3.9]{baishyaF}, we have $\frac{G}{Z(G)} \cong C_p \times C_p \times C_p \times C_p$.

Conversely, if $\mid \frac{G}{Z(G)} \mid =p^4$, then $n=p^2+2$ by noting that in the present scenario $C(a) \cap C(b)=Z(G)$ for any $a, b \in G \setminus Z(G)$ with $ab \neq ba$ and each proper centralizer of $G$ contains exactly $p^2$ distinct right cosets of $Z(G)$.
\end{proof}

We now have the following result which improves \cite[Theorem B]{jaa4}. It may be mentioned here that $\mid\frac{G}{Z(G)}\mid$ is well known for any $n (n \leq 10)$-centralizer group (see \cite[Theorem 3.5]{non} and use its arguments in the proof to main Theorems in \cite{baishya1, rostami}). 

\begin{thm}\label{sb1}
Let $G$ be any non-abelian $n$-centralizer group. If $n \leq  12$, then $\mid \frac{G}{Z(G)}\mid \leq (n-2)^2$  and $\mid \frac{G}{Z(G)}\mid \leq 2(n-4)^{{log}_2^{(n-4)}}$  otherwise.
\end{thm}

\begin{proof}
Suppose $n \leq  12$. If $G$ is not an F-group, then in view of  Lemma \ref{np1}, $G$ has an irredundant covering by $ \leq 10$ maximal abelian subgroups and consequently, the result follows using \cite[p. 857]{tom} (note that $G$ is a union of the centers of the proper centralizers of $G$). On the otherhand, if $G$ is  an an F-group, then the result follows using Theorem \ref{bc1}.

Second part follows from Corollary \ref{1sb} by noting that $(n-2)^2<2(n-4)^2< 2(n-4)^{{log}_2^{(n-4)}}$ for any $n>12$.
\end{proof}

As a consequence of Theorem \ref{bc1} we also have the next two results:

\begin{prop}\label{np2b}
Let $G$ be a finite $n$-centralizer group of conjugate type $(p, 1)$, $p$ a prime. Then $\mid\frac{G}{Z(G)}\mid  \leq (n -2)^2$, with equality iff $\frac{G}{Z(G)} \cong C_p \times C_p$.
\end{prop}

\begin{proof}

Note that $G$ is an F-group and using Proposition \ref{z-class5}, $ \lbrace \frac{Z(x_i)}{Z(G)} \mid 1 \leq i \leq n-1 \rbrace$ is a partition of $\frac{G}{Z(G)}$. Therefore in view of Theorem \ref{bc1}, we have $\mid\frac{G}{Z(G)}\mid  \leq (n -2)^2$. 

Now, suppose $\mid\frac{G}{Z(G)}\mid  = (n -2)^2$. Note that using \cite[Proposition 1]{mann}, we have 
$\mid \frac{Z(x_i)}{Z(G)} \mid=p$ for any $ 1 \leq i \leq n-1$. Moreover, by Theorem \ref{bc1}, we have $\mid \frac{Z(x_j)}{Z(G)} \mid \geq \surd \mid\frac{G}{Z(G)} \mid$ for some $ 1 \leq j \leq n-1$. Consequently, we have $\frac{G}{Z(G)} \cong C_p \times C_p$. Conversely, if $\frac{G}{Z(G)} \cong C_p \times C_p$, then $n=p+2$ (i.e., $\mid\frac{G}{Z(G)}\mid  = (n -2)^2$) by noting that in the present scenario $C(a) \cap C(b)=Z(G)$ for any $a, b \in G \setminus Z(G)$ with $ab \neq ba$ and each proper centralizer of $G$ contains exactly $p$ distinct right cosets of $Z(G)$.
\end{proof}

\begin{prop}\label{np2a}
Let $G$ be a finite $n$-centralizer group of conjugate type $(p^2, 1)$, $p$ a prime. Then $\mid\frac{G}{Z(G)}\mid  \leq (n -2)^2$, with equality iff $\frac{G}{Z(G)} \cong C_p \times C_p \times C_p \times C_p$.
\end{prop}

\begin{proof}

Note that $G$ is an F-group and using Proposition \ref{z-class5}, $ \lbrace \frac{Z(x_i)}{Z(G)} \mid 1 \leq i \leq n-1 \rbrace$ is a partition of $\frac{G}{Z(G)}$. Therefore in view of Theorem \ref{bc1}, we have $\mid\frac{G}{Z(G)}\mid  \leq (n -2)^2$. 

Now, suppose $\mid\frac{G}{Z(G)}\mid  = (n -2)^2$. Note that using \cite[Proposition 1]{mann}, we have $\mid \frac{Z(x_i)}{Z(G)} \mid \leq p^2$ for any $1 \leq i \leq n-1$. Moreover, by Theorem \ref{bc1}, we have $\mid \frac{Z(x_j)}{Z(G)} \mid \geq \surd \mid\frac{G}{Z(G)} \mid$ for some $1 \leq j \leq n-1$. Consequently, we have $\mid\frac{G}{Z(G)}\mid  \leq p^4$. In the present scenario, using \cite[Proposition 2.14]{baishya2}, we have $\mid\frac{G}{Z(G)}\mid  = p^4$ by noting that $G$ is of conjugate type $(p^2, 1)$. Now, using \cite[Proposition 3.9]{baishyaF}, we have $\frac{G}{Z(G)} \cong C_p \times C_p \times C_p \times C_p$.

Conversely, suppose $\mid \frac{G}{Z(G)} \mid=p^4$. Then $n=p^2+2$ (i.e., $\mid\frac{G}{Z(G)}\mid  = (n -2)^2$) by noting that in the present scenario $C(a) \cap C(b)=Z(G)$ for any $a, b \in G \setminus Z(G)$ with $ab \neq ba$ and each proper centralizer of $G$ contains exactly $p^2$ distinct right cosets of $Z(G)$.
\end{proof}

A finite $p$-group ($p$ a prime) $G$ is said to be a special $p$-group of rank $k$ if $G'=Z(G)$ is elementary abelian of order $p^k$ and $\frac{G}{G'}$ is elementary abelian. Furthermore, a finite group $G$ is extraspecial if $G$ is a special $p$-group and $\mid G' \mid=\mid Z(G) \mid=p$. A finite $p$-group ($p$ a prime) $G$ is semi-extraspecial if for every maximal subgroup $N$ in $Z(G)$ the quotient $\frac{G}{N}$ is extraspecial. In view of \cite[Remark 3.13]{baishyaF}, we note that every semi-extraspecial group is a special group  and for such groups $\mid C(x) \mid$ is equal to the index $[G: G']$ for all $x \in G \setminus G'$. A group $G$ is said to be ultraspecial if $G$ is semi-extraspecial and $\mid G' \mid= \sqrt{\mid G: G' \mid}$. It is known that for each prime $p$ there are $p+3$ ultraspecial groups of order $p^6$ and all of the ultraspecial groups of order $p^6$ (including for $p=2$) are isoclinic (see \cite[Remark 3.13]{baishyaF}).  

In view of the above discussions, using Proposition \ref{1np} and Theorem \ref{bc1}, we have the following property of a semi-extraspecial $p$-group.

\begin{prop}\label{semi}
Let $G$ be a finite $n$-centralizer semi-extraspecial $p$-group, $p$ a prime. Then $p \mid (n-2)$ and $\mid\frac{G}{Z(G)}\mid  \leq (n -2)^2$.
\end{prop}

\begin{rem}\label{bbu}
For $n \leq 10$, Tomkinson in \cite[p. 857]{tom} showed that  if $G$ is covered by $n$ abelian (maximal abelian) subgroups, then  $ \mid\frac{G}{Z(G)}\mid  \leq (n -1)^2$. He noted that the bound $(n-1)^2$ is attained by the non-abelian groups of order $p^3$, $p$ a prime. Moreover, he mentioned that he did not have examples for other values of $n$. It may be mentioned here that any group $G$ with $\frac{G}{Z(G)} \cong C_p \times C_p$ can be covered by $p+1$ abelian subgroups. Moreover, in view of Proposition \ref{np2a}, we can see that if $G$ is a finite $(n+1)$-centralizer group of conjugate type $(p^2, 1)$, $p$ a prime and $\mid \frac{G}{Z(G)} \mid=p^4$, then $G$ is covered by $n$ abelian subgroups (note that in this case $G$ is a CA-group) and $ \mid\frac{G}{Z(G)}\mid = (n -1)^2$. For example, if $G$ is any ultraspecial group of order $p^6$, $p$ a prime, then $G$ is of conjugate type $(p^2, 1)$ and $\mid \frac{G}{Z(G)} \mid=p^4$. Therefore from the proof of Proposition \ref{np2a}, $G$ is covered by $n(=p^2+1)$  abelian subgroups (centralizers) and  $\mid\frac{G}{Z(G)}\mid  = (n -1)^2$.
\end{rem}

Recall that a group $G$ is semi-simple if $G$ has no non-trivial normal abelian subgroup. The author in \cite[Proposition 2.5]{zarrin0942} proved that if $G$ is a semi-simple $n$-centralizer group, then  $\mid G \mid \leq (n-1)!$. Using Corollary \ref{1sb}, we can improve this result as follows: [note that max $ \lbrace (n-2)^2, 2(n-4)^{{log}_2^{(n-4)}} \rbrace < (n-1)!$ (see \cite{jaa4})].

\begin{cor}\label{bbc}
If $G$ is an $n$-centralizer non-abelian group, then $\mid \frac{G}{Z(G)} \mid < (n-1)!$.  
\end{cor}

We now give the following result concerning  bounds of $\mid \Cent(G)\mid$ for a finite centerless group. For the bounds of $\mid \Cent(G)\mid$ of an arbitrary finite group, see \cite[Theorem 3.3]{baishyaF}. In the following result  $\alpha (I(G))$ denotes the number of centralizers of $G$ produced by the elements of $I(G)$, where $I(G)=\lbrace x \in G \mid x=x^{-1}\rbrace$.

\begin{prop}\label{np12}
Let $G$ be a finite centerless group and $q$ be the largest prime divisor of its order. Then 
\begin{enumerate}
	\item $\mid \Cent(G)\mid \geq  q+2$, with equality iff $G=C_q \rtimes C_n$ is a Frobenius group. 
	\item $\mid \Cent(G)\mid \leq \mid G \mid-1$, with equality  iff $G \cong S_3$.
\end{enumerate} 

 \end{prop}

\begin{proof}

a) Suppose $x \in G$ is an element of order $q$ and let $a \in G \setminus C(x)$. Clearly, $ax^i \in G \setminus C(x)$ for any $i$. Consider the set $A=\lbrace x, a, ax, ax^2, \dots, ax^{q-1} \rbrace$. Observe that if $ax^iax^j=ax^jax^i$ for some $0 \leq i < j \leq q-1$, then $a \in C(x^{j-i})=C(x)$, a contradiction (noting that $gcd((j-i), o(x))=1$). Therefore $A$ is a set of pairwise non-commuting elements of $G$ and $\mid A \mid=q+1$. Hence $\mid \Cent(G)\mid \geq  q+2$. 

Now, suppose $\mid \Cent(G)\mid = q+2$. Then the proper centralizers of $G$ are precisely the centralizers given by the elements of $A$. Let $C(u)$ be non-abelian for some $u \in A$. Let $v \in C(u) \setminus Z(u)$. Then $C(v)=C(w)$ for some $w \in A$ and consequently, $u \in C(v)=C(w)$, which is a contradiction. Therefore $G$ is a CA-group. Now, if $G$ is non-solvable, then by \cite[Lemma 3.9]{abc},  $G \cong PSL(2, 2^m), 2^m>3$. But then using \cite[Theorem 1.1]{zarrin094}, $\mid \Cent(G)\mid \neq  q+2$. Therefore $G$ is solvable and consequently, by\cite[Proposition 3.1.1, Proposition 1.2.4]{elizabeth}, $G$ is a Frobenius group with cyclic complement. In the present scenario, the number of Frobenius complements is $q$ and consequently, by \cite{frob}, the Frobenius kernel is $C_q$. Hence the result follows. Conversely, if  $G=C_q \rtimes C_n$ is a Frobenius group, then $G$ has trivial center. Moreover, by the property of Frobenius groups, $q$ is the largest prime divisor of $\mid G \mid$ and  $\mid \Cent(G)\mid =  q+2$. \\

b) Clearly, $G$ must have an element $x$ such that $x \neq x^{-1}$. Therefore $\mid \Cent(G)\mid \leq \mid G \mid-1$ by noting that $C(x)=C(x^{-1})$.

Now, if $\mid G \mid$ is odd, then $o(y)=p>3$ for some $y \in G \setminus Z(G)$ and some prime $p$. But then 
$\mid \Cent(G)\mid < \frac{\mid G \mid}{2}+1<\mid G \mid-1$, a contradiction. Hence $\mid G \mid$ is even. In the present scenario, it is easy to see that $\mid G \setminus \I(G) \mid=2k$ for some $k$ and consequently, we have $\mid \I(G)\mid= \mid G \mid-2k$. Since $C(a)=C(a^{-1})$ for any $a \in G$, therefore we have  
\[
\mid \Cent(G)\mid \leq \frac{\mid G \setminus \I(G)\mid}{2}+ \mid \I(G)\mid=\mid G \mid -k \leq \mid G \mid-1.
\]
It now follows that $k=1$ and hence $\mid \I(G) \mid=\mid G \mid-2=\alpha(\I(G))$. In the present scenario, it is easy to verify that $C(b)=\langle b \rangle$ for any $b \in G \setminus Z(G)$. Hence $G \cong S_3$. Conversely, we have $\mid \Cent(S_3)\mid =\mid S_3 \mid-1$. 
\end{proof}

 Let  $q$ be the largest prime divisor of the order of a group $G$. The authors in \cite[Theorem 6]{ctc092} proved that $\frac{\mid \Cent(G)\mid}{\mid G \mid} \leq \frac{1}{2}$, if $q=2$ and $\frac{\mid \Cent(G)\mid}{\mid G \mid} \leq \frac{3}{4}+ \frac{1}{4q}$ otherwise. Accordingly, they raised the question whether or not there exists a finite group other than $Q_8$ and $D_{2p}$ ($p$ a prime) such that $\frac{\mid \Cent(G)\mid}{\mid G \mid} \geq \frac{1}{2}$. In \cite{en09}, the author answered negatively that question by showing that $\mid \Cent(D_{2n}) \mid=n+2$, if $n$ is odd ($D_{2n}$ denotes the dihedral group of order $2n$). In this connection, we prove the following result which gives a new property of extraspecial $2$-groups. This  also gives yet another two families of counterexamples  to the above question.
 
Recall that a finite group $G$ is extraspecial if $G$ is a special $p$-group and $\mid G' \mid=\mid Z(G) \mid=p$. It is well known that every extraspecial $p$-group has order $p^{2a+1}$ for some positive integer $a$. Moreover, for each prime $p$ and for every positive integer $a$, there exists, upto isomorphism, exactly two extraspecial groups of order $p^{2a+1}$.

\begin{prop}\label{t1}
If $G$ is a finite group with non-trivial center, then $\mid \Cent(G)\mid \leq \frac{\mid G \mid}{2}$, with equality iff $G$ is an extraspecial $2$-group.
\end{prop}

\begin{proof}
Since $C(x)=C(xz)$ for any $x \in G$ and any $z \in Z(G)$, therefore $\mid \Cent(G)\mid \leq \frac{\mid G \mid}{2}$.

Now, suppose $\mid \Cent(G)\mid = \frac{\mid G \mid}{2}$. Then we have  $\mid Z(G) \mid=2$ by noting that  $C(x)=C(xz)$ for any $x \in G$ and any $z \in Z(G)$. In the present scenario, using \cite[Proposition 3.1]{baishyaF}, we have $\frac{G}{Z(G)}$ is an elementary abelian $2$-group. Therefore $G'=Z(G)$ and consequently, $G$ is an extraspecial $2$-group. Conversely, if $G$ is an extraspecial $2$-group, then by \cite[Proposition 3.13]{baishyaF}, $\mid \Cent(G)\mid = \frac{\mid G \mid}{2}$.
\end{proof}

It may be mentioned here that the alternating group $A_4$ of degree $4$ has trivial center and $\mid \Cent(A_4) \mid=\frac{\mid A_4 \mid}{2}$. In the present scenario, the following problem is worth mentioning:

\begin{prob}
Classify the finite groups $G$ with $\mid \Cent(G)\mid = \frac{\mid G \mid}{2}$.
\end{prob}

We now classify finite F-groups with $\mid \Cent(G)\mid \geq \frac{\mid G \mid}{2}$.

\begin{thm}\label{thm1}
$G$ is  a finite F-group with $\mid \Cent(G)\mid \geq \frac{\mid G \mid}{2}$ iff $G \cong A_4 $, an extraspecial $2$-group or a Frobenius group with abelian kernel and complement of order $2$. 
\end{thm}

\begin{proof}
Let $G$ be a finite F-group. Suppose $\mid \Cent(G)\mid = \frac{\mid G \mid}{2}$. Now, if $\mid Z(G)\mid >1$, then by Proposition \ref{t1}, $G$ is an extraspecial $2$-group. On the other hand if $\mid Z(G)\mid =1$, then using classification of groups of rank 1 (see \cite[p.578]{zappa}), $G$ is a CA-group. Now, if $G$ is non-solvable, then by \cite[Lemma 3.9]{abc},  $G \cong PSL(2, 2^m), 2^m>3$ and using \cite[Theorem 1.1]{zarrin094}, $\mid \Cent(G)\mid < \frac{\mid G \mid}{2}$, which is a contradiction. Therefore $G$ is solvable and consequently, by\cite[Proposition 3.1.1, Proposition 1.2.4]{elizabeth}, $G\cong K \rtimes H$ is a Frobenius group with abelian kernel $K$ and cyclic complement $H$. Now, using properties of Frobenius groups we have $\mid \Cent(G)\mid = \mid K \mid+2= \frac{\mid K \mid \mid H \mid}{2}$. It now follows that $\mid H \mid=3$ and $\mid K \mid=4$. Therefore $\mid G \mid=12$ and hence $G \cong A_4$.

Next, suppose $\mid \Cent(G)\mid > \frac{\mid G \mid}{2}$. We have $C(x)=C(xz)$ for any $x \in G$ and any $z \in Z(G)$ and consequently, $\mid Z(G) \mid=1$. Now, using classification of groups of rank 1 (see \cite[p.578]{zappa}), $G$ is a CA-group. Suppose  $G$ is non-solvable. Then by \cite[Lemma 3.9]{abc},  $G \cong PSL(2, 2^m), 2^m>3$ and using \cite[Theorem 1.1]{zarrin094}, $\mid \Cent(G)\mid < \frac{\mid G \mid}{2}$, which is a contradiction. Therefore $G$ is solvable and consequently, by \cite[Proposition 3.1.1, Proposition 1.2.4]{elizabeth}, $G\cong K \rtimes H$ is a Frobenius group with abelian kernel $K$ and cyclic complement $H$. In the present scenario, using the properties of Frobenius groups, we have $\mid \Cent(G)\mid = \mid K \mid 
+2$. Now, suppose $\mid H \mid \geq 3$. 
Then $ \frac{\mid G \mid}{2}=\frac{\mid K \mid \mid H \mid }{2} \geq \mid K \mid + \frac{\mid K \mid}{2} > \mid K \mid +2$, which is a contradiction by noting that we must have $\mid K \mid >4$ in this situation. Therefore we have $H \cong C_2$.

Conversely, if $G \cong A_4$, an extraspecial $2$-group or a Frobenius group with abelian kernel and complement of order $2$, then $G$ is an F-group (see classification of groups of rank 1 \cite[p.578]{zappa}. Clearly,  $\mid \Cent(A_4)\mid = 6$. Moreover,  if $G$ is an extraspecial $2$-group, then by Proposition \ref{t1}, $\mid \Cent(G)\mid = \frac{\mid G \mid}{2}$. On the other hand, if $G$ is a Frobenius group with abelian kernel and complement  of order $2$, then from the properties of Frobenius groups, we have $\mid \Cent(G)\mid = \frac{\mid G \mid}{2}+2$.
\end{proof}

\begin{cor}\label{ccor1}
 $G$ is  a finite CA-group with $\mid \Cent(G)\mid \geq \frac{\mid G \mid}{2}$ iff $G \cong A_4, Q_8, D_8$ or a Frobenius group with abelian kernel and complement of order $2$. 
\end{cor}

\begin{proof}
The result follows from Theorem \ref{thm1} by noting that if $G$ is an extraspecial $2$-group of order $> 8$, then $G$ is not a CA-group (see classification of groups of rank 1 \cite[p.578]{zappa}). 
\end{proof}

In view of Theorem \ref{bc1} and Theorem \ref{thm1}, we now have the following result:

\begin{prop}\label{xx}
If $G$ is  a finite $n$-centralizer non-abelian F-group, then  $\mid \frac{G}{Z(G)}\mid \leq (n-2)^2 \leq \frac{{\mid G \mid}^2}{4}$.
\end{prop}

There exists finite groups, for example, say $G=S_3 \times S_3$ such that $\mid \Cent(G )\mid > \frac{\mid G \mid}{2}+2$. However, it follows from Theorem \ref{thm1} that for any finite F-group $G$, we have  $\mid \Cent(G)\mid \leq \frac{\mid G \mid}{2}+2$, with equality iff  $G$ is a Frobenius group with abelian kernel and  complement  of order $2$. It may be mentioned here that $\mid \Cent (S_4) \mid=\frac{\mid S_4 \mid}{2}+2$ but $S_4$ is not an F-group. In this connection it is natural to study the following problem:

\begin{prob}
Classify the finite groups $G$ with $\mid \Cent(G)\mid = \frac{\mid G \mid}{2}+2$.
\end{prob}

Let $H$ be any subgroup of a group $G$. It is easy to see that $\mid \Cent(H) \mid \leq \mid \Cent(G)\mid$. The authors in \cite{con} studied some conditions for the equality and among other results  they  showed that (\cite[Proposition 3.6]{con}) if $G$ is a finite group such that $\frac{G}{Z(G)}$ is isomorphic to a simple group, then $G$ and $G'$ are isoclinic groups  and consequently, $\mid \Cent(G) \mid = \mid \Cent(G')\mid$ (for basic notions of isoclinism see \cite{pL95, non}). Here we give a generalization of this result (a group $G$ is called perfect if $G=G'$):

\begin{prop}\label{CG118}
Let $G$ be an $n$-centralizer non-abelian group. Then $\frac{G}{Z(G)}$ is perfect if and only if $G$ is isoclinic with $G'$. In particular, $\mid \Cent(G) \mid = \mid \Cent(G')\mid$.
\end{prop}

\begin{proof}
In view of  \cite[Theorem 1.2 (Baer)]{tom}, $\frac{G}{Z(G)}$ is finite by noting that $G=Z(x_1)\cup \dots \cup Z(x_{n-1})$. Suppose $\frac{G}{Z(G)}$ is perfect. Then $\frac{G}{Z(G)}=(\frac{G}{Z(G)})'=\frac{G'Z(G)}{Z(G)}$ and consequently, $G'Z(G)=G$. Therefore in view of \cite[Lemma 2.7]{pL95}, $G$ is isoclinic to $G'$. Conversely, if if $G$ is isoclinic with $G'$, then using \cite[Lemma 2.7]{pL95} we have  $G'Z(G)=G$, which implies $\frac{G}{Z(G)}=(\frac{G}{Z(G)})'$.  Moreover, by \cite[Lemma 3.2]{non} we have $\mid \Cent(G) \mid = \mid \Cent(G')\mid$.
\end{proof}

Finally, we conclude the paper with a family of groups which are F-groups but not CA-groups. In this connection, we want to point out that the author in \cite[Proposition 1.3]{brough} proved that an extraspecial $p$-group ($p$ a prime) of order  $> p^3$ is an F-group which is not a CA-group. However, this is already a known result which was initially obtained by Zappa (see \cite[pp. 578]{zappa}). The following proposition extends this result and  gives a large family of F-groups which are not CA-groups. Note that every extraspecial $p$-group is of conjugate type $(p, 1)$.

\begin{prop}\label{Za1}
Let $G$ be a finite group of conjugate type $(p^k, 1)$ for some prime $p$. If $\mid \frac{G}{Z(G)} \mid > p^{2k}$, then $C(x)$ is non-abelian for any $x \in G \setminus Z(G)$.
\end{prop}

\begin{proof}
Following Ito \cite{ito}, without any loss we may assume that $G$ is a $p$ group. Moreover, in view of \cite[Proposition 1]{mann}, we have $\mid \frac{Z(x)}{Z(G)} \mid \leq p^k$ for any
 $x \in G \setminus Z(G)$. Hence the result follows.

\end{proof}



\end{document}